\documentclass{article}
\title{Computing weight one modular forms over $\C$ and $\Fpbar$.}
\author{Kevin Buzzard} 
\usepackage{amsmath} 
\usepackage{amssymb} 
\newcommand{\Qbar}{\overline{\Q}}
\newcommand{\C}{\mathbf{C}}
\newcommand{\Q}{\mathbf{Q}}

\newcommand{\Qlbar}{\Qbar_\ell}
\newcommand{\Z}{\mathbf{Z}}

\newcommand{\F}{\mathbf{F}}

\newcommand{\Fbar}{\overline{\F}}
\newcommand{\Fpbar}{\Fbar_p}
\newcommand{\calO}{\mathcal{O}}
\newcommand{\GQ}{\Gal(\Qbar/\Q)}
\newcommand{\GM}{\Gal(\overline{M}/M)}
\newcommand{\new}{\mbox{new}}
\newcommand{\rhobar}{\overline{\rho}}
\newcommand{\Xbar}{\overline{X}}

\DeclareMathOperator{\Gal}{Gal}
\DeclareMathOperator{\GL}{GL}
\DeclareMathOperator{\PGL}{PGL}
\DeclareMathOperator{\PSL}{PSL}
\DeclareMathOperator{\SL}{SL}
\DeclareMathOperator{\Ind}{Ind}
\DeclareMathOperator{\disc}{disc}
\DeclareMathOperator{\Frob}{Frob}
\DeclareMathOperator{\Spec}{Spec}
\usepackage{amsthm}
\theoremstyle{plain}
\newtheorem{theorem}{Theorem}
\newtheorem{lemma}[theorem]{Lemma}
\newtheorem{corollary}[theorem]{Corollary}
\newtheorem{proposition}[theorem]{Proposition}
\theoremstyle{remark}
\newtheorem{remark}[theorem]{Remark}

\begin{document}
\maketitle 
\begin{abstract}
We report on a systematic computation of weight one cuspidal eigenforms
for the group $\Gamma_1(N)$ in characteristic zero and in characteristic $p>2$.
Perhaps the most surprising result was the existence of a mod~$199$ weight~1 cusp form
of level~82 which does not lift to characteristic zero.
\end{abstract}

\section*{Introduction}
It is nowadays relatively easy to compute spaces of classical cusp forms
of weights two or more, thanks to programs by William Stein written
for the computer algebra packages Magma~\cite{magma} and SAGE~\cite{sage}.
On the other hand, there seems to be relatively little published regarding
explicit computations
of weight one cusp forms. In characteristic zero, computations have
been done by Buhler
(\cite{buhler}) and Frey and his coworkers (\cite{freyetal}), and
there is a beautiful paper of Serre (\cite{serre}) which
explains several tricks for computing with weight one forms
of prime level, but as far as
we know there is (until now) nothing systematic in the literature. 
In characteristic~$p$ (where there are sometimes more forms -- that is,
forms that do not lift to characteristic zero) there is even less
in the literature; see~\cite{wt1edix} and the references therein,
and the beginning of section~3 of this paper, for more information.
The algorithms of~\cite{wt1edix} for computing mod~$p$ forms
of weight~1 have been implemented by Wiese in the computer algebra
package Magma. Note however that they involve
giving both~$N$ and~$p$ as inputs, and the running time depends on~$p$.

In this paper
we report on a fairly systematic computation of weight~1 forms that we did using
Magma about ten years ago now (although we re-ran the code more recently
and, unsurprisingly, the same programs could now go a little further).
The characteristic zero code we wrote has since
been incorporated as part of Magma, but the methods work just as well in
characteristic~$p$; indeed for a given level~$N$ we can compute mod~$p$ for
all odd primes $p\nmid N$ at once. In particular our algorithm takes as
input only~$N$, and spits out a basis for the characteristic zero
weight~1 forms, plus the primes~$p$ for which there are mod~$p$
forms that do not lift, plus the $q$-expansions of these non-liftable
forms. We remark that we do not even attempt to define mod~$p$
modular forms of level~$N$ if $p\mid N$. We apologise that it has taken
so long to get the results down on paper.
Our methods and calculations have in the mean
time been greatly extended by George Schaeffer, and his forthcoming
thesis~\cite{schaeffer:thesis}
contains many many more examples of mod~$p$ forms that do not lift to
characteristic zero.

The methods basically go back to Buhler's thesis~\cite{buhler}, and the main
idea is very simple: we cannot compute in weight~1 directly using modular
symbols, but if we choose a non-zero modular form~$f$ of weight $k\geq1$
then multiplication by~$f$ takes us from cusp forms
of weight~1 to cusp forms of weight $k+1\geq2$ where
we can compute using modular symbols, and then we divide by~$f$ again to produce
the space of weight~1 forms with possible poles where $f$ vanishes.
Repeating this idea for lots of choices of~$f$ and intersecting
the resulting spaces will often enable us to
compute the space of holomorphic weight~1 forms rigorously. We explain
the details in the next section.

Recent developments by Khare and Wintenberger on
Serre's conjecture give another approach for computing
weight~1 forms in characteristic zero: instead of working on
the automorphic side one can compute on the Galois side.
The work of Khare and Wintenberger
implies that there is a canonical bijection between the set of weight~1
normalised new cuspidal eigenforms over $\C$
and the set of continuous odd irreducible representations $\GQ\to\GL_2(\C)$
(see Th\'eor\`eme~4.1 of~\cite{deligne-serre} for one direction
and Corollary~10.2 of~\cite{KW1} for the other).
Say $\rho_f$ is the Galois representation attached to the form~$f$; then
the level of~$f$ is the conductor of~$\rho_f$. Let us consider for a moment
an arbitrary irreducible representation $\rho:\GQ\to\GL_2(\C)$. The projective
image of $\rho$ in $\PGL_2(\C)$ is a finite subgroup of $\PGL_2(\C)$ and is
hence either cyclic, dihedral, or isomorphic to $A_4$, $S_4$ or $A_5$, and in
fact the cyclic case cannot occur because $\rho$ is irreducible. If~$f$ is a
characteristic zero eigenform then we say that the \emph{type} of~$f$ is
dihedral, tetrahedral ($A_4$), octahedral ($S_4$) or icosahedral ($A_5$)
according to the projective image of $\rho_f$.
Because of the Khare--Wintenberger
work, one approach for computing weight~1 modular forms of level~$N$
in characteristic zero would be to list all the finite extensions of $\Q$
which could possibly show up as the kernel of the projective representation
attached to a level~$N$ form, and then reverse-engineer the situation
(carefully analysing liftings and ramification, which would not be much
fun) to produce the forms themselves. Listing the extensions is just
about possible: one can use class
field theory to deal with the dihedral cases, and the $A_4$, $S_4$ and $A_5$
calculations can be done because the projective representation attached to a
level~$N$ form is unramified outside~$N$, so to compute in level~$N$
one just has to list all $A_4$, $S_4$ and $A_5$ extensions unramified
outside~$N$ (which is feasible nowadays for small~$N$ and indeed
there are now tables of such things, see for example~\cite{jrtables}
for an online resource). Note that $A_4$ and $S_4$ are solvable, so one
could perhaps in theory also use class field theory to analyse these
cases. However lifting the projective representation
to a representation can be troublesome to do in practice. Furthermore,
this approach does not work in characteristic~$p$: the problem is
if $k$ is a finite subfield of $\Fpbar$ then $\PGL_2(k)$ is a finite
subgroup of $\PGL_2(\Fpbar)$, and this gives us infinitely many more
extensions which must be checked for; hence the method breaks down.
In fact looking for mod~$p$ Galois representations with large image
is quite hard, it seems, and perhaps it is best to do the calculations
on the automorphic side and use known theorems to deduce results
on the Galois side. For example, as a consequence of our search at
level~82 we proved the following result:
\begin{theorem}

(a) There is a mod~199 weight~1 cusp form of level~82 which is not
the reduction of a weight~1 characteristic zero form.

(b) There is a number field~$M$, Galois over~$\Q$, unramified
outside~2 and~41, with Galois group~$\PGL_2(\Z/199\Z)$.
\end{theorem}
This result, which relies on a computer calculation, is proved
in the final section of this paper.

When we initially embarked upon this computation, the kinds of things we
wanted to know were the following:

\begin{itemize}

\item What are the ten (or so) smallest
integers~$N$ for which the space of weight~1 cusp forms of level~$N$
is non-zero?

\item What is the smallest~$N$ for which there exists a weight~1
level~$N$ eigenform whose associated Galois representation has
projective image isomorphic to $A_4$? To $S_4$?

\item Give some examples of pairs $(N,p)$ consisting of an integer~$N$
and a prime~$p$ for which there is a mod~$p$ weight~1 eigenform
of level~$N$ which does not lift to a characteristic zero weight~1 eigenform
of level~$N$ (Mestre had already given an example with $(N,p)=(1429,2)$;
see the appendices to~\cite{wt1edix}; can one find examples with $p>2$?).

\end{itemize}

We found it hard to extract the answers to these questions from the
literature, so we answered them ourselves with a systematic computation
of weight one forms in characteristic zero and $p>2$. In~2002 we looked
in the range $1\leq N\leq 200$ in characteristic zero, and $1\leq N\leq 82$
in characteristic~$p$
(adopting a ``quit while you're ahead'' policy in the characteristic~$p$
case).
Both of these bounds are very modest and even ten years ago, when we
actually did the calculations, it would have been possible to go further.
When Gabor Wiese re-ignited our interest in this project we ran the
characteristic zero programs once more on a more modern computer, and this time
they ran much faster, and up to $N=352$, before running out of memory.
George Schaeffer has since stepped up to the plate and his forthcoming
PhD thesis pushes these calculations much further.

Our systematic computations did not find any characteristic zero weight~1
modular forms with associated Galois representations
having projective image isomorphic to $A_5$ so we do not know what
the smallest conductor of an $A_5$-representation is. Some trickery
computing mod~5 representations attached to weight~5 forms did enable
us to ``beat Buhler's record'' however -- so we do now know that
there is an icosahedral weight~1 form of level~675; Buhler's weight~1
form has level~800. We verified that there were no newforms of level $N\leq 352$
whose associated projective Galois representation has image isomorphic
to $A_5$. Hence the smallest level of a weight~1 form of $A_5$ type
is in the range $[353,675]$.\footnote{Note added in May 2016: explicit computations by Alan Lauder using a fast machine and the algorithms of this paper have proved that the smallest level of a weight~1 $A_5$ form is in fact~633; see forthcoming work ``A computation of modular forms of weight one and small level'' by Lauder and the author for more details.}  Note that by the remarks on p248 (case (c${}_2$))
of~\cite{serre}, and the computer calculations of~\cite{freyetal}
(see in particular
section~4.1 of~\cite{kiming}), we know that the smallest \emph{prime}
appearing as the conductor of an $A_5$ representation is~2083 (and
hence the range can be shortened to $[354,675]$).

The outline of this paper is as follows. We explain our algorithm in
section~1, summarise the characteristic zero
results in section~2, and the characteristic~$p$ results in section~3.

Finally we would like to thank the both Frank Calegari and
the anonymous referee of this paper, for helpful remarks.

\section{Computing weight one cusp forms.}

We remind the reader of some definitions. If $N\geq5$ is an integer, then
there is a smooth affine curve $Y_1(N)$ over $\Z[1/N]$
parameterising elliptic curves over $\Z[1/N]$-schemes equipped with
a point of exact order~$N$. The fibres of $Y_1(N)\to\Spec(\Z[1/N])$
are geometrically
irreducible. If $E_1(N)$ denotes the universal elliptic
curve over $Y_1(N)$ then the pushforward of $\Omega^1_{E_1(N)/Y_1(N)}$
is a sheaf $\omega$ on $Y_1(N)$. There is a canonical compactification
$X_1(N)$ of $Y_1(N)$, obtained by adding cusps, and this curve is
smooth and proper over $\Z[1/N]$. Furthermore the sheaf $\omega$ extends
in a natural way
to $X_1(N)$. If $R$ is any $\Z[1/N]$-algebra then we denote by $X_1(N)_R$
the pullback of $X_1(N)$ to~$R$. We write $M_k(N;R)$ for
$H^0(X_1(N)_R,\omega^{\otimes k})$ and refer to this space
as the level~$N$ weight~$k$ modular forms defined over~$R$.
We write $S_k(N;R)$ for the sub-$R$-module of this $R$-module consisting
of sections which vanish at every cusp.

If $K$ is a field where $N$ is invertible
then the $K$-dimension of $M_k(N;K)$ is~0 if $k<0$,
it is~1 if $k=0$,
and can be easily computed if $k\geq2$ using the Riemann--Roch formula. If however $k=1$
then the Riemann--Roch theorem unfortunately only tells us the dimension of the
subspace of Eisenstein series in $M_k(N;K)$. Similarly
for $k\not=1$ the dimension of $S_k(N;K)$ is also easily computed,
but for $k=1$, even the dimensions of these spaces seem to
lie deeper in the theory. 

A weight one cusp form of level $N$ is a section of $\omega$ which
vanishes at every cusp, and is hence a section of $\omega\otimes C^{-1}$ where
$C$ is the sheaf associated to the divisor of cusps.
One can compute the degree of $\omega$ on $X_1(N)_{\C}$ without too
much trouble: for example $\omega^{\otimes 12}$ descends to a degree~1
sheaf on the $j$-line $X_0(1)_{\C}$ and hence the degree of $\omega$
on $X_1(N)$ is $[\SL_2(\Z):\Gamma_1(N)]/24$.
Similarly the number of cusps on $X_1(N)_{\C}$ is well-known to be
$\frac{1}{2}\sum_{0<d|N}\phi(d)\phi(N/d)$, the sum being over the positive
divisors of $N$. Hence the degree of $\omega\otimes C^{-1}$ is
easily computed in practice for small~$N$.
Although we did these calculations over $\C$, the degree of $\omega$
is the same in characteristic zero and in characteristic~$p$ (we are
assuming $N\geq5$ so the scheme of cusps over $\Z[1/N]$ is etale).
From these formulae, which are messy but entirely elementary,
it is easy to deduce the following.

\begin{lemma} Let $K$ be a field in which the positive integer
$N$ is invertible.
Then there are no non-zero weight~1 cusp forms of level $N$ over $K$,
if $5\leq N\leq 22$, $24\leq N\leq 28$, $N=30$ or $N=36$.
\end{lemma}

\begin{proof} Indeed, the degree of $\omega\otimes C^{-1}$ is less than
zero in these cases.
\end{proof}

If $N\leq 4$ then there are theoretical issues with the approach we
have adopted, because $X_1(N)_K$ is only a coarse moduli space,
and there is no natural sheaf $\omega$ on $X_1(N)_K$ (there is a
problem at one of the cusps when $N=4$, and problems at elliptic points when
$N\leq3$). On the
other hand, one can still give a rigorous definition of a
modular form of level $N$ for $N\leq 4$ (using the theory of algebraic
stacks, for example, or the classical definition as functions on the
upper half plane if $K=\C$) and one easily checks, using any of
these definitions, that a cusp form of
level $N$ is also naturally a cusp form of level $Nt$ for any positive
integer $t$. Because 1,2,3 and 4 all divide~12, and there are no
non-zero cusp forms of level~12 by the above lemma, we conclude

\begin{corollary} Let~$K$ be a field where the positive integer~$N$
is invertible.
There are no non-zero cusp forms of level~$N$ over~$K$ for any~$N<23$.
\end{corollary}

The same argument shows that the dimension of the space of weight~1
cusp forms of level~23 is at most~1, because the degree of
$\omega\otimes C^{-1}$
on $X_1(23)$ is~0. Moreover, there will be a non-zero cusp form of level~23
if and only if this sheaf is isomorphic to the structure sheaf
on $X_1(23)$. 
Conversely, there is indeed a non-zero level~23 weight~1 cusp form
in characteristic zero: namely the form $\eta(q)\eta(q^{23})$,
where $\eta=q^{1/24}\prod_{n\geq1}(1-q^n)$.
The mod~$p$ reduction of this form
is a mod~$p$ cusp form for any $p\not=23$, and this proves that the
dimension of the level~23 forms is~1 in characteristic zero and
in characteristic~$p\not=23$.

We have now solved the problem of computing weight one level~$N$
cusp forms for $N\leq28$, but of course such tricks
only work for small levels, and for $N\geq29$
we used a computer to continue our investigations. Our strategy was
as follows. Let~$K$ be an algebraically closed field where~$N\geq29$
is invertible. Let $S_k(N;K)$ denote the weight~$k$ cusp forms of level~$N$
defined over~$K$. We wish to compute $S:=S_1(N;K)$.
First let us choose a form $0\not=f\in M_k(N;K)$ for some $k\geq1$
that we can compute the $q$-expansion of to arbitrary
precision (for example $f$ can be a form of weight at least~2,
or a weight~1 Eisenstein series or theta series).
Then $f.S:=\{fh:h\in S\}$ is a subspace of $S_{k+1}(N;K)$,
which is a space that we can compute as $k+1\geq2$.
We compute
a basis of $q$-expansions for $S_{k+1}(N;K)$, and
then divide each $q$-expansion by~$f$, giving us an explicit
finite-dimensional space of $q$-expansions which contains~$S$.
Repeating this for many choices of~$f$ and continually taking
intersections will typically cut this space down, but after a while
its dimension will stabilise. Let~$V$ be the space of $q$-expansions
so obtained; this is now our candidate
for~$S$. We know for sure that it contains~$S$.
In fact, if we could somehow choose forms~$f_1$ and~$f_2$ as above,
which were guaranteed to have no zeros in common on $X_1(N)_K$,
then we would know for sure that our space really was~$S$.
However, we know of no efficient way of testing to see whether
two given forms share a zero on $X_1(N)$, especially in
characteristic~$p$. Note that in~\cite{freyetal}, working
over $\C$, a careful choice of $f_1$ and $f_2$ is indeed made,
to guarantee that they have no common zero;
our approach is more haphazard.

So far, we have a ``candidate space'' of $q$-expansions, which we know
includes~$S$ and this gives us a reasonable upper bound for the dimension
of~$S$. To get a good lower bound, because we were only really interested
in the case of~$N$ at most~350 or so, we wrote a program
which counts dihedral representations (the logic being the folklore
conjecture that ``most weight~1 forms are dihedral'').
More precisely, what our program does is the
following. For a given~$N$ it counts the number of
representation $\GQ\to\GL_2(\C)$ which are continuous, odd, irreducible,
induced from a character of a quadratic extension of~$\Q$, and have
conductor~$N$. This
computation is a finite one because if $M$ is a quadratic extension
of $\Q$ and $\psi:\GM\to\C^\times$ is a continuous 1-dimensional
representation then the conductor of $\Ind(\psi)$ is the absolute
value of $\disc(M)|c(\psi)|$, where $c(\psi)$ is the conductor of $\psi$
(an ideal of the integers of~$M$), and $|c(\psi)|$ is its norm.
Hence there are only finitely many possibilities for~$M$ and, for
each possibility, one can use class field theory to enumerate the
characters $\GM\to\C^\times$ of conductor $I$, for $I$ any ideal
of $\calO_M$. This approach gives a lower bound for the dimension
of the space of newforms in $S$, and if we repeat
the computation for divisors of $N$, then we get a lower bound for
the dimension of $S$.

We have explained how to get both upper and lower bounds for the dimension
of a space~$S$ of characteristic zero weight one cusp forms. If these bounds
coincide, which of course they often do in practice in the range we
considered,  then
we have computed the dimension of~$S$, we have also proved that
the Galois representations associated to all eigenforms of level~$N$
and conductor~$\chi$ are induced from characters of index two subgroups,
and furthermore, because
both methods we have sketched are constructive, we now have two
ways of actually computing the $q$-expansions of a basis for~$S$
to as many terms as we like, within reason -- an automorphic method
and a Galois method.

If however we run the algorithm above, and the lower bound it produces
is still strictly less than our upper bound, then we guess that there
are some non-dihedral
forms at level~$N$ that are not contributing to our lower bound.
What we now need is a way of rigorously proving that these formal
$q$-expansions really do correspond to holomorphic weight~1
forms rather than forms with poles. We do this as follows.
Choose a form~$h$ in our vector space~$V$. We are now suspecting
that~$h$ is holomorphic; what we know is that $h=g/f$ for some
non-zero weight~$k$ form~$f$ and weight~$k+1$ form~$g$.
Let~$D$ denote the divisor of zeros of~$f$. Then $h$ is a meromorphic
section of $\omega$, with divisor of poles bounded by~$D$.
In particular we have a bound on the degree
of the divisor of poles of $h^2$, which is now meromorphic and
weight~2. Now here's the trick. If we can find a \emph{holomorphic}
weight~2 form~$\phi$ of level~$N$ whose $q$-expansion is the same as that of
$h^2$ up to order $q^{M+1}$, where $M$ is a large integer,
then we have \emph{proved} that $h^2$ is holomorphic;
for $h^2-\phi$ is a weight~2 form which is a holomorphic section
of $\omega^{\otimes2}\otimes(2D)\cong\omega^{\otimes (2+2k)}$ and yet
it has a zero of order at least~$M$
at $\infty$, so as long as $M$ is greater than the degree of
$\omega^{\otimes (2+2k)}$
the form $h^2-\phi$ must be identically zero, and in particular~$h$
must be holomorphic. If we can prove that a basis for~$V$ consists
of holomorphic forms, then we have proved $V=S$. If this algorithm
fails then we have really proved $V\not=S$ and we go back to
choosing forms~$f$ as above and dividing out. Eventually in practice
the process terminates, at least for $N\leq 350$ or so on architecture
that is now ten years old. As mentioned before, 
George Schaeffer has taken all of this much further now.

In practice we do not quite do what is suggested above. Firstly,
instead of working with the full space of forms of level~$N$
we fix a Dirichlet character of level~$N$ and work with forms
of level~$N$ and this character. This gives us a huge computational
saving because it cuts down the dimension of all the spaces
we are working with by a factor of (very) approximately~$N$.
It does introduce some thorny issues at primes dividing $\phi(N)$,
where the diamond operators may not be semisimple, but these
can be dealt with by simply gritting one's teeth and ignoring
diamond operators whose order divides~$p$ in this case.
In fact the issues became sufficiently thorny here for $p=2$
that we decided to leave $p=2$ alone and restrict to the case $p>2$.

Secondly, in fact we do not work over a field at all; we do the entire
calculation on the integral level, working over $\Z[\zeta_n]$
for $\zeta_n$ a primitive $n$th root of unity, $n$ chosen
sufficiently large that all the relevant Dirichlet characters showing
up in the computation have order dividing~$n$. In characteristic zero
we only need to compute with one Dirichlet character per Galois conjugacy
class; but a characteristic zero conjugacy class can break into
several conjugacy classes mod~$p$ and so we need to reduce things
not modulo~$p$ but modulo the prime ideals of~$\Z[\zeta_n]$.
If $\chi$ and $\alpha$
are $\Z[\zeta_n]$-valued Dirichlet characters of level~$N$, and we are trying
to compute in level~$N$, weight~1 and character~$\chi$, then we choose
$f\in S_k(N,\alpha;\Z[\zeta_n])$
and let $L$ denote the lattice $S_{k+1}(N,\alpha\chi;\Z[\zeta_n])/f$.
We run through many choices of~$f$ and, instead of intersecting
vector spaces, we intersect lattices. When computing an intersection
of two lattices $L_1$ and $L_2$ arising
in the above way, one computes not just the intersection but
also the size of the torsion subgroup of $(L_1+L_2)/L_1$;
if any prime number divides the order of this torsion subgroup
then the intersection of $L_1$ and $L_2$ is bigger in characteristic~$p$
than in characteristic~zero. The torsion subgroup is a $\Z[\zeta_n]$-module
and we compute the primes above~$p$ in its support; the reduction of $\chi$
modulo these prime ideals are the mod~$p$ characters where there may
be more mod~$p$ forms than characteristic zero forms. 
Note that in practice the order of the torsion subgroup of $(L_1+L_2)/L_1$
can be so big that it is unfactorable, but this does not matter
because we simply collect all the orders of these torsion groups
and continually compute their greatest common divisor. What often
happens in practice is that we manage to prove that the intersection
is no bigger in characteristic~$p$ than in characteristic zero
for all odd $p\nmid N$; then we have proved that all characteristic~$p$
forms lift to characteristic zero.

Occasionally however we may run into a
prime number~$p$ which shows up in these torsion orders to the extent
that we cannot rule out the dimension of the mod~$p$ space being higher;
we can then compute the $q$-expansion of a candidate non-liftable
form and square it and look in weight~2 in characteristic~$p$
as explained above; if we can
find the $q$-expansion of the square in weight~2 to sufficiently high
precision then we have constructed a non-liftable form.

These tricks, put together, always worked in the region in which we
did computations, which was $N\leq352$ in characteristic~0 and
$N\leq82$ in characteristic~$p>2$. We stopped at $N=352$ in characteristic
zero because of memory issues; by then
we had seen $A_4$ and $S_4$ extensions, but we still felt a long
way from finding an $A_5$ example -- there
is no reason why one should not be able to proceed further by using
a more powerful modern machine.
In characteristic~$p$ we were running
into problems of factoring very large integers when computing
the torsion subgroups of the quotients above, so we stopped at $N=82$ because
of a very interesting example that we found there (see section~3).

\section{Characteristic zero results.}

We ran our calculations in characteristic~0 for all $N\leq352$. Of course,
the dimension of $S_1(N,\chi;\C)$ was often zero.

\subsection{Small level.}

In characteristic zero there is a good theory of oldforms and newforms,
and we firstly list the dimensions of all the non-zero spaces
$S_1^{\new}(N,\chi;\C)$ for $N\leq60$. Note that if $\chi_1$
and $\chi_2$ are Galois conjugate characters then the associated
spaces $S_1^{\new}(N,\chi_1;\C)$ and $S_1^{\new}(N,\chi_2;\C)$
are also Galois conjugate in a precise sense, and in particular have
the same dimension, so we only list characters up to
Galois conjugacy. The (lousy) notation we use for characters is as
follows: if the prime factorization of~$N$ is $p^eq^f\ldots$, 
then a Dirichlet character $\chi$ of level~$N$ can be written
as a product of Dirichlet characters $\chi_p$, $\chi_q\ldots$
of levels $p^e$, $q^f$,\ldots.
By $p_a$ we mean a character $\chi_p$ of level $p^e$ and order~$a$,
and by $p_aq_b\ldots$ we mean the product of such characters.
This notation will not always specify the Galois conjugacy
class of a character uniquely (which is why it's lousy), but it does
in the cases below apart from the case $N=56$, where we need to add
that the character $2_2$ is the unique even character of level~8
and order~2 (thus making the product $2_27_2$ odd). 

\smallskip

\begin{tabular}{|c|c|c|}
\hline
$N$& $\chi$& dimension of $S_1(N,\chi;\C)$\\
\hline
23& $23_2$&     1\\
31& $31_2$&     1\\
39& $3_2 13_2$& 1\\
44& $2_1 11_2$& 1\\
47& $47_2$&     2\\
52& $2_2 13_3$& 1\\
55& $5_2 11_2$& 1\\
56& $2_2 7_2$&  1\\
57& $3_2 19_3$& 1\\
59& $59_2$&     1\\
\hline
\end{tabular}

\smallskip

One can now deduce, for example, that the dimension of $S_1(52;\C)$ is
two, because no $N$ strictly dividing~52 appears in the table,
at $N=52$ the character in the table above has a non-trivial
Galois conjugate, and both the character and its conjugate contribute~1
to the dimension. Similar computations give the dimensions of all
weight~1 spaces of level $N\leq 60$. All of these forms are of dihedral type
and hence are easily explained in terms
of ray class groups. For example, the two newforms at level~47 are explained
by the fact that the class group of $L=\Q(\sqrt{-47})$ is cyclic
of order~5, and if $H$ denotes the Hilbert class field
of $L$ then the four non-trivial characters of $\Gal(H/L)$
can all be induced up to give 2-dimensional Galois representations
of $\Gal(H/\Q)$ of conductor~47 (one gets two isomorphism classes
of 2-dimensional representations).
In particular, the two newforms of level~47 are defined over
$\Q(\sqrt{5})$ and are Galois conjugates. As another example,
one checks that if $P$ is a prime above~13 in~$\Q(i)$ then
the corresponding ray class field of conductor~$P$ has degree~3
over $\Q(i)$, and the corresponding order~3 character
of the absolute Galois group of $\Q(i)$ can
be induced up to $\GQ$ giving a 2-dimensional representation of conductor~52
which is readily checked to be irreducible and odd, and is the
representation corresponding to the level~52 form in the table above
(up to Galois conjugacy). 

One can of course also explain all the forms in the table above
using theta series: for example the first
form in the list has level~23
and quadratic character; the corresponding normalised newform $f$ can be
written down explicitly: $2f=\sum_{m,n}q^{m^2+mn+6n^2}-q^{2m^2+mn+3n^2}$.
There is also another well-known formula for $f$, namely
$f=\eta(q)\eta(q^{23})$, where $\eta(q)=q^{1/24}\prod_n(1-q^n)$.
For other examples one can see~\cite{serre}.

\subsection{$A_4$ examples.}

Our algorithm computed lower bounds for spaces of forms by counting
dihedral representations, and hence our methods make it easy to spot when
one has discovered a form which is not of dihedral type. To work out
what is going on with these forms one needs to do both local and
global calculations; the more pedantic
amongst us might at this point like to choose algebraic closures~$\Qbar$
of~$\Q$, and~$\Qlbar$ of~$\Q_\ell$ for all primes~$\ell$, and
also embeddings of~$\Qbar$ into~$\Qlbar$ (for all~$\ell$) and
into $\C$; this makes life slightly easier in terms of notation.

Notation: if $\ell$
is a prime then $D_\ell$ denotes the absolute Galois group of
$\Q_\ell$, and $I_\ell$ is its inertia subgroup. 

The smallest level where there is a non-dihedral form is $N=124=2^2\times31$.
In fact at level~124 there are four non-dihedral newforms
(and no other newforms, although there is a 3-dimensional space of oldforms
coming from level~31). Let~$\chi$ be a level~124 Dirichlet
character with order~2 at~2 and order~3 at 31; then
$S_1(124,\chi;\C)$ has dimension~2, as does $S_1(124,\chi^{-1};\C)$
(note that $\chi^{-1}$ is the complex conjugate of~$\chi$). What are
the Galois representations attached to the corresponding eigenforms?
Well, let $f,g$ denote the two normalised eigenforms in $S_1(124,\chi;\C)$.
\begin{lemma} The projective image of the Galois representations
associated to $f$ and $g$ are isomorphic to $A_4$; furthermore,
in both cases the number field cut out by this projective
representation is the splitting field~$K$ of $x^4 + 7x^2 - 2x + 14$.
\end{lemma}
\begin{proof}
We use the following strategy. We first construct two odd Galois
representations to $\GL_2(\C)$ of conductor 124 and determinant $\chi$,
whose projective images both cut out~$K$; such representations
are known to come from weight~1 forms of level~124 and hence
they must be the representations associated to~$f$ and~$g$.
The lemma is hence reduced to the construction of these two Galois
representations, the heart of the matter being controlling the
conductor. The strategy for doing such things was already used
heavily in~\cite{freyetal},
and the key input is Theorem~5 of~\cite{serre}, a theorem
of Tate, which states the following: 
if we have a projective representation
$\rhobar:\GQ\to\PGL_2(\C)$ and for each ramified prime $\ell$ with
associated decomposition group $D_\ell$ we choose
a lifting of $\rhobar|D_\ell$ to $\rho_\ell:D_\ell\to\GL_2(\C)$, then
there is a global lift $\rho:\GQ\to\GL_2(\C)$ of $\rhobar$
such that if $I_\ell$ is the inertia subgroup of $D_\ell$
then $\rho|I_\ell\cong\rho_\ell|I_\ell$. In particular
the conductor of $\rho$ is the product of the conductors
of the $\rho_\ell$. This result reduces the computation of conductors
of global lifts to a local calculation, which we now do.

The splitting field~$K$ has degree~12 over~$\Q$.
Let $\rhobar:\Gal(K/\Q)\to\PGL_2(\C)$
be an injection. We will lift $\rhobar$ to $\rho$ with conductor~124.
One easily checks
using a computer algebra package that~$K$ is unramified outside~2 and~31.
The decomposition group at~2 is cyclic of order~2 and the completion
of~$K$ at a prime above~2 is isomorphic to $\Q_2(\sqrt{3})$.
The decomposition group at~31 is cyclic of order~3 and cuts
out a ramified degree~3 extension $K_{31}$ of $\Q_{31}$.
The projective representation on the decomposition group at~2
lifts to a reducible representation $1\oplus\tau$, with
$\tau$ the order~2 local character which cuts out $\Q_2(\sqrt{3})$.
This local representation has conductor $4$.
Similarly the projective representation at~31 lifts
to the reducible representation $1\oplus\sigma$, with $\sigma$
an order~3 character with kernel corresponding to~$K_{31}$.
This local representation has conductor~31.
Tate's theorem now implies that there is a global lift
$\rho:\GQ\to\GL_2(\C)$ of $\rhobar$ with conductor~124, which
furthermore on $I_2$ looks like $1\oplus\tau$. We know that $\rho$
is not induced from an index~2 subgroup of $\GQ$ (if it were then
the projective image of $\rho$ would be dihedral), and hence
$\rho':=\rho\otimes\chi_4\not\cong\rho$, where $\chi_4$ is the conductor~4
Dirichlet character. A local calculation at~2 shows that $\rho'$
also has conductor~4 (note that $\Q_2^{nr}(\sqrt{3})=\Q_2^{nr}(\sqrt{-1})$).
Furthermore, $\rho$ and $\rho'$ are odd (because $K$ is not totally real)
and hence both modular, so correspond to two distinct newforms
of level~124; these forms must be~$f$ and~$g$.
\end{proof}

Note that the strategy of the above proof easily generalises to
other levels, assuming one can find the relevant number field,
which we could do in every case that we tried simply by looking
through tables of number fields of small degree unramified outside
a given set of primes. 
The next level where we see non-dihedral forms is $N=133=7.19$,
with character~$\chi$ of order~2 at 7 and~3 at 19; again
this determines $\chi$ up to conjugation.
The weight~1 forms and the corresponding representations
were originally discovered by Tate and
some of his students in the 1970s; see the concluding remarks
of~\cite{tate:hilbert} for some more historical information
about the calculations (which were done by hand)\footnote{We thank
Chandan Singh Dalawat for pointing out this reference to us on MathOverflow.}.
Again the dimension of the space of weight~1
forms of level~$N$ and character~$\chi$ is~2, both forms
are of $A_4$ type and the corresponding $A_4$-extension of~$\Q$
is the splitting field of $x^4 + 3x^2 - 7x + 4$. This can be proved
using the same techniques as the preceding lemma; the splitting
field is unramified outside~7 and~19, the decomposition group at~7
is cyclic of order~2 cutting out $\Q_7(\sqrt{7})$
and the decomposition group at~19 is cyclic of order~3; 
the global quadratic character one twists by to get the second
form is the one with conductor~7.

Whilst we did not go through tables of extensions of $\Q$ carefully
and do the analogues of the above calculations to check everything
rigorously, computational results seemed to indicate
that the first few levels for which there
are~$A_4$ newforms are the levels 124, 133, 171, 201, 209, 224\ldots .

\subsection{$S_4$ examples.}

We talked about the non-dihedral forms of levels~124 and~133 in the
previous section. 

The next non-dihedral newform occurs at level~$148=2^2\times37$, and it is
our first form of type~$S_4$. If $\chi$ is a Dirichlet
character of level~148 which is trivial at~2 and has order~4
at~37 (there are two such characters, and they are Galois conjugate)
then the dimension
of $S_1(148,\chi;\C)$ is~1. One checks using similar
techniques that the $S_4$-extension of $\Q$ cut out by
the projective Galois representation must be
the splitting field of~$x^4 - x^3 + 5x^2 - 7x + 12$.
Indeed, the splitting field of this polynomial has Galois group $S_4$,
is unramified outside~2 and~37, the decomposition
group at~2 is isomorphic to $S_3$ with inertia
the order~3 subgroup, and the decomposition and inertia
groups at~37 are both cyclic of order~4. Note that any
inclusion $S_3\to\PGL_2(\C)$ lifts to an inclusion $S_3\to\GL_2(\C)$;
the induced map $\rho_2:D_2\to\GL_2(\C)$ has conductor~4
because it is the sum of two order three tame characters on inertia.

The author confesses that he was initially slightly surprised to see this
latter extension show up again when looking for mod~$p$ phenomena
in the next section.

Again, an analysis of tables of $S_4$ extensions of~$\Q$
should enable one to check that the first
few levels for which there are $S_4$ forms
are the levels~148, 229, 261, 283, 296\ldots.

\subsection{An $A_5$ example.}

Our search for forms of level $N$, $1\leq N\leq 352$, did not
reveal any $A_5$ forms. However following a suggestion of the anonymous
referee, we did discover the following.

\begin{proposition} There is a level~675 characteristic zero
weight~1 form
whose associated Galois representation has projective
image isomorphic to $A_5$.
\end{proposition}

\begin{proof} We present a proof which is in a sense unenlightening.
Let~$L$ be the splitting field of $x^5 - 25x^2 - 75$. Using a computer
algebra package (for example Magma) one explicitly checks
that $L$ is not totally real, that $\Gal(L/\Q)\cong A_5$, and
that~$L$ is unramified outside~3 and~5. Furthermore 3 factors
as $(P_1P_2\ldots P_{10})^6$ with each $P_i$ having degree~1,
the discriminant of~$L$ at~3 is $3^{70}$, and hence
the discriminant of the completion $L_1$ of $L$ at~$P_1$ is $3^7$.
Hence, if $G_i$ denote the lower numbering filtration on the
group $G_0=\Gal(L_1/\Q_3)$ we know $7=\sum_i(|G_i|-1)$,
and $|G_0|=6$, $|G_1|=3$, so $G_n=1$ for all $n\geq2$.
From this we conclude that the conductor of a minimal lift
of the associated representation $A_5\to\PGL_2(\C)$ is $3^3$ at~3.
A similar calculation at~5 shows that the conductor of a minimal
lift is~$5^2$. By Khare--Wintenberger the associated minimal lift Galois
representation $\Gal(L/\Q)\to\GL_2(\C)$ has conductor $3^35^2=675$
and is modular.
\end{proof}

As is probably clear, the proof does not answer the most important
question, namely how one runs across the polynomial $x^5-25x^2-75$.
The answer is that, following a suggestion of the anonymous referee, we
computed mod~5 forms of small weight and character until we got
lucky. At level~27 there is an ordinary weight~5 cuspidal
eigenform for which the associated mod~5 Galois representation seemed
to have image contained in $Z.SL(2,\F_5)$ with $Z$ the centre of $GL(2,\F_5)$.
A search through the online tables~\cite{jrtables}, knowing that there
may be an $A_5$ extension of~$\Q$ unramified outside~3 and~5 for which
the associated weight~1 form may have conductor less than~800,
soon led us to
the polynomial. Note in particular that the discovery of this $A_5$ form
was completely independent of the weight~1 computations described
in the rest of this paper.

\section{Unliftable mod~$p$ weight~1 forms.}

Also perhaps of some interest are the mod~$p$ forms of level~$N$ that do
not lift to characteristic~0 forms of level~$N$. We first note
the following subtlety: there is a mod~3 form of level~52 with
character of order~2 at~2 and trivial at~13, and this mod~3 form
does not lift to a characteristic zero cusp form with order~2 character.
But this is not surprising -- indeed this phenomenon can happen
in weight $k\geq2$ as well, and in weight 2 it first happens at $N=13$;
this was the reason that Serre's initial predictions about
the character of the form giving rise to a modular representation
needed a slight modification. The weight~1 mod~3 form of level~52
in fact lifts to a level~52 form with character of order~6, and also
to an Eisenstein series of level~52 with order~2 character; on the
Galois side what is happening is that the weight~1
level~52 cusp form of dihedral type mentioned in the previous section has
associated Galois representation which is irreducible
and induced from an index~2 subgroup, but the mod~3 reduction of
this representation is reducible. The phenomenon of not being able
to lift characters in an arbitrary manner was deemed
``uninteresting'' and we did not explicitly search for it (it happens
again mod~5 at level~77 and many more times afterwards; note in particular
that it can happen mod~$p$ for $p\geq5$). 

We now restrict our attention to mod~$p$ weight~1 eigenforms of level~$N$
which do not lift to characteristic zero eigenforms of level~$N$. We did
an exhaustive search for such examples with $p>2$. The first
example we found was at level~74, where there is a mod~3 form
with character trivial at~2 and of order~4 at~37. The associated
mod~3 Galois representation was checked to be irreducible and
have solvable image, and this confused the author for a while,
because he was under the impression that the only obstruction to lifting
weight~1 forms was lifting the image of Galois. This notion is indeed
vaguely true, but what is happening here is that the mod~3 form
of level~74 lifts to no form of level~74 but
to the form of type $S_4$ and level $148=2\times 74$ which we described
in the previous section. Why has the conductor gone up?
It is for the following reason: there is a 2-dimensional mod~3
representation of $D_2$ whose image
is the subgroup
$\bigl(\begin{smallmatrix}1&*\\0&*\end{smallmatrix}\bigr)$ of $\GL_2(\F_3)$
and such that the image of inertia has order~3. The conductor of this mod~3
representation is~$2^1$; however all lifts of this representation
to $\GL_2(\C)$ have conductor at least $2^2$. This example
was nice to find, firstly because it gave the author confidence
that his programs were working, but secondly it somehow really
emphasizes just how
miraculous this whole theory is -- these subtleties of conductors
dropping show up on both the automorphic side and the Galois side.

After finding this level~74 example, what we now realised we really
wanted was a mod~$p$ form which
did not lift to any weight~1 form at all. Fortunately we soon
found it -- it was at level $N=82=2\times 41$, and to our surprise
was a mod~199 form (note that 199 is prime) whose associated Galois
representation had rather large image.
We finish this note by explaining what we found here.

Let $\F_{199^2}$ denote the field with $199^2$ elements. Fix a root
$\tau$ of $X^2+127X+1$; changing $\tau$ will just change everything
below by the non-trivial field automorphism of $\F_{199^2}$.
One can check that the multiplicative order of~$\tau$ is~40.

Let $\chi$ be the group homomorphism $(\Z/82\Z)^\times\to\F_{199^2}^\times$
which sends $47\in(\Z/82\Z)^\times$ (note that 47 is a generator of
the cyclic group $(\Z/82\Z)^\times$) to~$\tau$.
Our programs showed that the space of mod~199 weight~1 cusp forms of level~$N$
and character~$\chi$ was 1-dimensional. Let~$f$ denote this weight~1
eigenform. The reader who wants to join in at home will need
to know the first few terms in the $q$-expansion of~$f$:

\begin{align*}
f=&q + (18\tau + 85)q^2 + (183\tau + 55)q^3 + (120\tau + 135)q^4 + (171\tau + 45)q^5\\
& +     (187\tau + 187)q^6 + (140\tau + 128)q^7 + (194\tau + 161)q^8 + (84\tau + 141)q^9\\
&    + (151\tau + 150)q^{10} + (106\tau + 4)q^{11} + (127\tau + 191)q^{12} + (112\tau +    92)q^{13}\\
& + (27\tau + 2)q^{14} + (146\tau + 37)q^{15} + (172\tau + 44)q^{16} + (192\tau     + 4)q^{17}\\
&+ (137\tau + 125)q^{18} + (189\tau + 117)q^{19} + O(q^{20})
\end{align*}

This is enough to determine~$f$ uniquely. For one can formally multiply
this power series by a weight~1 Eisenstein series of level~82
and character $\chi^{-1}$ and the resulting $q$-expansion (which at this
point we know up to $O(q^{20})$) is the $q$-expansion of a
level~82 mod~199 weight~2 form with trivial character which turns out to be
determined uniquely by the first~19 coefficients of its~$q$-expansion.
The $q$-expansion of this unique weight~2 form
can be worked as far as one wants within reason; we computed
it to $O(q^{20000})$ in just a few minutes. Dividing through by
the Eisenstein series again gives us the $q$-expansion of~$f$ to as
much precision as one wants.

Computing the $q$-expansion of~$f$ to high precision gives us, for free,
plenty of facts about the mod~199 Galois representation associated
to~$f$. Note first
that there \emph{is} a mod~199 Galois representation attached to~$f$;
one cannot use the Deligne--Serre theorem here (because $f$ doesn't lift
to a weight~1 form of characteristic zero)
but one can construct the representation by general theory as follows.
One can multiply~$f$ by the mod~199 Hasse invariant~$A$ to get a mod~199
form~$Af$ of weight~199 which is an eigenvector for all Hecke operators
away from~199; the smallest Hecke-stable subspace containing~$Af$
is 2-dimensional and consists of two eigenvectors $f_1$ and $f_2$;
the $T_\ell$-eigenvalues of~$f$, $f_1$ and $f_2$ all coincide
for $\ell\not=199$, and the $T_{199}$-eigenvalues of $f_1$ and $f_2$
are the two roots of $X^2-a_{199}X+\chi(199)$, which are distinct.
Both $f_1$ and $f_2$
lift to characteristic zero weight~199 eigenforms which are ordinary at~199,
and the associated mod~199 Galois representations are isomorphic;
this is the mod~199 Galois representation $\rho_f$ associated to~$f$.

One consequence of the computation of~$f$ is that that its $q$-expansion
has all
coefficients in $\F_{199^2}$. Write $f=\sum_{n\geq1}a_nq^n$, with
$a_n\in\F_{199^2}$. Explicit computation shows that $a_2$ and $a_{41}$
are non-zero. The Brauer group of a finite field is trivial
and hence there is an associated semisimple Galois representation
$\rho_f:\GQ\to\GL_2(\F_{199^2})$. Standard facts about the mod~$p$
Galois representations associated
to modular forms now imply that $\rho_f$ is unramified outside $\{2,41\}$.
Indeed, for $\ell\not=199$ this is standard, and for $\ell=199$
we use the fact that the mod~$199$ representations of $D_{199}$ 
attached to $f_1$
and $f_2$ above are upper triangular with unramified characters on the
diagonal, and furthermore the unramified character showing up
as the subspace in $\rho_{f_1}$ is the character that shows up
as the quotient in $\rho_{f_2}$; hence $\rho_f$ restricted to $D_{199}$
is a direct sum of two unramified characters, $\rho_f(\Frob_{199})$
is semisimple, and the characteristic
polynomial of $\rho_f(\Frob_{199})$ is $X^2-a_{199}X+\chi(199)$,
just as the characteristic polynomial of $\rho_f(\Frob_{\ell})$
is $X^2-a_{\ell}X+\chi(\ell)$ for all other primes $\ell\not\in\{2,41\}$.

Standard mod~$p$ local-global results apply to $f_1$ and $f_2$,
because they have weight bigger than~1 and lift to characteristic zero.
Furthermore standard level-lowering results apply to $\rho_f$, because
it is absolutely irreducible -- indeed, if it were reducible then
its semisimplification would be the sum of two characters of conductor
dividing~82, and hence if $\ell_1$ and $\ell_2$ were primes congruent
mod~82 then one would have $a_{\ell_1}=a_{\ell_2}$. However one checks
from the computations that $a_7\not=a_{89}$. In particular ``classical''
level-lowering results can be applied to $f_1$ and $f_2$, and one
can deduce level-lowering results for~$f$. For example, there are no
weight~1 mod~$199$
forms of level~41 and character~$\chi$, hence $\rho_f$ must be ramified at~2.
It is also ramified at~41 (because its determinant is).

The kernel of $\rho_f$ hence corresponds to a number
field~$L$ ramified only at~2 and~41 such that $\Gal(L/\Q)$ is
isomorphic to the image of $\rho_f$, which is a subgroup of $\GL_2(\F_{199^2})$.
The local representation at~2 is Steinberg, which means that
inertia at~2 has order~199; the local representation at~41 is
principal series corresponding to one unramified and one tamely
ramified (but ramified) character, which implies that inertia at~41
has order~40. In particular~$L$ is tamely ramified at both~2 and~41.

The only natural question left is: what is the image of $\rho_f$,
or equivalently what is $\Gal(L/\Q)$?
For several years we assumed that the image would contain
$\SL_2(\F_{199^2})$, but it was only in~2009 that we actually
tried to sit down and compute it, and we discovered that in
fact the image is smaller.

\begin{proposition} The image~$X$ of $\rho_f$ is, after conjugation
in $\GL_2(\overline{\F}_{199})$ if necessary, contained
in $Z\cdot\GL_2(\F_{199})$, with $Z$ the subgroup
of scalar matrices in $\GL_2(\F_{199^2})$.
Furthermore the quotient of~$X$ by $X\cap Z$ is $\PGL_2(\F_{199})$.
\end{proposition}

\begin{corollary} 

(i) There is a number field~$M$, Galois over~$\Q$,
unramified outside~2 and~41, tamely ramified at~2 and~41, and
with $\Gal(M/\Q)=\PGL_2(\F_{199})$.

(ii) The weight~1 form~$f$ does not lift to any weight~1 eigenform
of characteristic zero.
\end{corollary}
\begin{proof} (of corollary) (i) $M$ is the kernel of $\overline{\rho}_f$.
(ii) There is no finite subgroup of $\GL_2(\C)$ with a subquotient
isomorphic to the simple group $\PSL_2(\F_{199})$ and so $\rho_f$
does not lift to $\GL_2(\C)$; hence neither does~$f$.
\end{proof}
\begin{remark} Our original proof of the proposition involved
computing many of the $q$-expansion coefficients of~$f$; indeed
one crucial intermediate calculation took a month. We are very
grateful to Frank Calegari for suggesting a simpler approach. Our
original approach can be found on the original ArXiv posting
of this article, in case anyone else is wondering what it is.
\end{remark}

\begin{proof} (of Proposition) Let $\overline{\rho}_f$ denote the
projective representation associated to $\rho_f$,
so $\overline{\rho}_f:\GQ\to\PGL_2(\F_{199^2})$.
It suffices to prove that the image of $\overline{\rho}_f$ 
is the subgroup $\PGL_2(\F_{199})$ of $\PGL_2(\F_{199^2})$.
We now adopt a rather brute force approach.
We have defined~$X$ to be the image of $\rho_f$; we now
let $\Xbar$ denote the image of $\overline{\rho}_f$ in $\PGL_2(\F_{199^2})$.
The finite subgroups of $\PGL_2(\overline{\F}_p)$ for $p$ a prime
were classified
by Dickson (see~\cite{dickson}, sections~255 and~260); they
are as follows. They are either conjugate in $\PGL_2(\overline{\F}_p)$
to a subgroup
of the upper-triangular matrices, are dihedral of order prime to~$p$,
are isomorphic to $A_4$, $S_4$ or $A_5$,
or are conjugate to $\PSL_2(k)$ or $\PGL_2(k)$ for some finite
subfield $k\subset\overline{\F}_p$. We will rule out all but
one possibility for~$\Xbar$; we know of no other way of proving
the result.

We start by observing that looking at the $q$-expansion of~$f$
gives upper and lower bounds for the size of~$\Xbar$.
First, the size of~$\Xbar$ is
bounded above by the size of $\PGL_2(\F_{199^2})$, and this means
that $\Xbar$ cannot be conjugate to $\PSL_2(k)$ or $\PGL_2(k)$
for any finite field $k$ of size at least $199^3$. In fact
because $\det(\rho_f)=Im(\chi)=\mu_{40}\subset\F_{199^2}$ consists
of squares in $\F_{199^2}$, the image of $\rho_f$ is contained
within $\mu_{80}\SL_2(\F_{199^2})$ (where here $\mu_{80}$ is the
cyclic group of $80$th roots of unity considered as a subgroup
of the scalars in $\GL_2(\F_{199^2})$), which rules out the case
$\Xbar=\PGL_2(\F_{199^2})$ as well.

On the other hand we can compute the semisimple conjugacy classes
of the first 1500
unramified primes by explicitly computing the $q$-expansion
of~$f$ to high precision; this only takes a few minutes
and shows that $\Xbar$ has
at least~199 elements (the point being that $a_{\ell}^2/\chi(\ell)$
takes on all~199 values of $\F_{199}$ as $\ell$ varies over
the first 1500 unramified primes; we do not need to worry about
whether Frobenius elements actually are semisimple, because if
$g\in X\subseteq\GL_2(\F_{199^2})$ is non-semisimple then its semisimplification
is $g^{199^2}\in X$). This means that $\Xbar$ cannot
be isomorphic to $A_4$, $S_4$ or $A_5$. Next we observe that~$\Xbar$ cannot be
conjugate to a subgroup of the upper-triangular matrices.
For if it were, the semisimple representation $\rho_f$ would be the sum of
two characters each having conductor a divisor of~82 and
in particular one could deduce that if $\ell_1$ and $\ell_2$
were unramified primes which were congruent mod~82 then
$a_{\ell_1}$ and $a_{\ell_2}$ would be equal; however this
is not the case, as $a_7\not=a_{89}$. As a consequence we
deduce that $\rho_f$ is absolutely irreducible.

We are left with the following possibilities: $\Xbar$ can be dihedral
of order prime to~199,
or conjugate to $\PSL_2(k)$ for $k$ of size~199 or~$199^2$, or conjugate
to $\PGL_2(\F_{199})$. We next rule out 
the dihedral case; if $\Xbar$ were dihedral then $\rho_f$ restricted
to an index two subgroup would be the sum of two characters,
and hence $\rho_f$ would be induced from an index~2 subgroup
corresponding to a quadratic extension of~$\Q$ unramified outside~2 and~41.
There are only seven such extensions, namely $\Q(\sqrt{D})$
for $D\in\{2,41,82,-1,-2,-41,-82\}$, and for each one it
is easy to find a prime~$\ell\not\in\{2,41\}$ which is inert in
the extension and such that $a_{\ell}\not=0$ (indeed the smallest
prime $\ell$ such that $a_{\ell}=0$ is $\ell=193$, but there is an inert
prime $p\leq 13$ in each of these quadratic extensions). However all such
Frobenius elements would have trace zero if $X$ were dihedral.

We next rule out the case $\Xbar=\PSL_2(\F_{199^2})$.
Our original method rather convoluted, involving writing $f=\sum_na_nq^n$
and verifying, on a computer, that the coefficients $a_n$ for all $n\leq 353011$
with $n\equiv1$~mod~41 were all in $\F_{199}$. We then argued
that $\sum_{t\geq0}a_{1+41t}q^{1+41t}$, a weight~1 form
of level~$2\times 41^2=3362$, must have all $q$-expansion coefficients
in $\F_{199}$, which gave information about the restriction of $\rho_f$ to 
the absolute Galois group of $\Q(\zeta_{41})$, and this sufficed.
This calculation took several weeks on a computer (and the details
can be found on the first ArXiv version of this paper). We are grateful
to Frank Calegari for providing the following alternative argument.
The trick is to observe that our computer calculations have
shown that the space of mod~199 level~82 weight~1
forms of character $\chi$ is 1-dimensional, and is spanned
by~$f$. Let us now construct another element of this space, as follows.
In simple terms it could be described as ``the Galois conjugate
of the newform attached to $f\otimes\chi^{-1}$'', but because the
theory of newforms is thorny in characteristic~$p$ let us spell
out how to construct this form explicitly. 
First take~$f_1$, the mod~199 weight~199 eigenform from before with the same
Galois representation as~$f$. Lift it to $F_1$, a characteristic zero eigenform
of weight~199 and level~82, and with character~$\tilde\chi$ of conductor~41.
The only possibility
for the local component of the automorphic representation attached to~$F_1$
at~41 is a principal series representation attached to one unramified
and one ramified character. This means that the newform~$G_1$ attached
to $F_1\otimes\tilde\chi^{-1}$ also has level~82. Similarly if we lift
$F_2$ to a form of character~$\tilde\chi$, the newform~$G_2$
attached to $F_2\otimes\tilde\chi^{-1}$ has level~82. Furthermore
the mod~199 reductions of $G_1$ and $G_2$ have level~82, character
$\chi^{-1}=\chi^{199}$ and $q$-expansions
which are equal apart from the coefficients of $q^n$ with $199\mid n$.
This means that their difference is the 199th power of a weight~1
form of level~82 and character~$\chi$. and this form must hence be
a multiple of~$f$.

Unravelling this we see that we have proved that if $f=\sum_na_nq^n$
then for all primes $\ell\not=2,41,199$ we have
$\overline{a}_\ell=a_\ell/\chi(\ell)$,
where $\overline{a}_\ell=a_\ell^{199}$ is the Galois conjugate of~$a_\ell$.
In particular if $x\in\Xbar$ and we lift $x$ to
$y\in\SL_2(\overline{\F}_{199})$ with eigenvalues $\alpha$ and $1/\alpha$
then $(\alpha+1/\alpha)^2=a_\ell^2/\chi(\ell)\in\F_{199}$. However it
is easy to check that most elements of $\PSL_2(\F_{199^2})$ do
not have this property. We conclude that $\Xbar\not=\PSL_2(\F_{199^2})$.

We finally have to distinguish between the two remaining possibilities
for $\Xbar$, namely $\Xbar=\PSL_2(\F_{199})$ and and $\Xbar=\PGL_2(\F_{199})$.
Because $\Xbar$ can be no larger than $\PGL_2(\F_{199})$ we do know
that (after conjugation if necessary) $X\subseteq Z.\GL_2(\F_{199})$,
with $Z$ the scalars in $\GL_2(\F_{199^2})$.
Furthermore $\det(X)=\mu_{40}\subset\F_{199^2}^\times$ and hence
$X\cap Z\subseteq\mu_{80}$. One checks that if $\ell$ is the
prime~661 then $\rho_f(\Frob_{\ell})$ has semisimplification a scalar
matrix with order~40 and hence $\mu_{40}\subseteq X\cap Z$. Furthermore
the normal index~40 subgroup~$Y:=\rho_f(\Gal(\Qbar/\Q(\zeta_{41})))$ of~$X$
is contained within $\SL_2(\F_{199})$
and hence $\overline{Y}=Y\cap Z$ is a normal subgroup of $\overline{X}$
of index at most~40 and hence a normal subgroup of $\PSL_2(\F_{199})$
of index at most~40. But $\PSL_2(\F_{199})$ is simple and hence
$\overline{Y}=\PSL_2(\F_{199})$. This means that $Y$ is either
$\SL_2(\F_{199})$ or an index~2 subgroup -- but $\SL_2(\F_{199})$
is a perfect group and hence $Y=\SL_2(\F_{199})$, and so
$\mu_{40}\SL_2(\F_{199})\subseteq X$. Because we know $Y$ has index~40
in~$X$ we deduce that $\mu_{40}\SL_2(\F_{199})$ is an index~2 subgroup
of~$X$. If $\Xbar=\PSL_2(\F_{199})$ then this forces $X=\mu_{80}\SL_2(\F_{199})$,
but this cannot be the case because the eigenvalues $\alpha$
and $\beta$ of $\rho_f(\Frob_3)$ have the following property:
if $\delta\in\mu_{80}\subset\F_{199^2}$ satisfies $\delta^2=\alpha\beta$
then $(\alpha+\beta)/\delta\not\in\F_{199}$. Hence $\Xbar=\PGL_2(\F_{199})$.
\end{proof}
We finish by remarking that the corresponding $\PGL_2(\F_{199})$-extension
of $\Q$ contains a quadratic field~$J$, corresponding to the subgroup
$\PSL_2(\F_{199})$. It is easy to establish
what this subextension is, as it is unramified outside~2 and~41
and in fact also unramified at~2, because $L/\Q$ is tamely
ramified at~2, and hence it must be $\Q(\sqrt{41})$.
In particular we deduce the existence of a Galois extension
of $\Q(\sqrt{41})$, unramified outside~2 and~41,
with Galois group $\PSL_2(\F_{199})$.

\bibliographystyle{amsalpha}
\newcommand{\etalchar}[1]{$^{#1}$}
\providecommand{\bysame}{\leavevmode\hbox to3em{\hrulefill}\thinspace}
\providecommand{\MR}{\relax\ifhmode\unskip\space\fi MR }
\providecommand{\MRhref}[2]{%
  \href{http://www.ams.org/mathscinet-getitem?mr=#1}{#2}
}
\providecommand{\href}[2]{#2}

\end{document}